\documentclass[12pt]{article}
\usepackage{amssymb,amsmath,enumerate,latexsym,epsf,graphicx}

\newfont{\bb}{msbm10 at 12pt}
\newfont{\bbp}{msbm10 at 9pt}

\def\d{\hbox{\bb D}}
\def\r{\hbox{\bb R}}
\def\rt{\hbox{\bbp R}}
\def\h{\hbox{\bb H}}

\def\mr{\hbox{\bb M}^2\times\hbox{\bb R}}
\def\hr{\hbox{\bb H}^2\times\hbox{\bb R}}

\def\man{\mathcal{M}}
\def\hm{\mathcal{M}(\kappa , \tau)}
\def\hmf{\hbox{\bb E}(\kappa , \tau)}
\def\m{\hbox{\bb M}^2}

\def\c{\hbox{\bb C}}
\def\s{\hbox{\bb S}}

\newcommand{\camb}{\overline{\nabla}}

\newcommand{\norm}[1]{\left\Vert #1 \right\Vert}
\newcommand{\abs}[1]{\left\vert #1 \right\vert}
\newcommand{\set}[1]{\left\{#1\right\}}
\newcommand{\meta}[2]{\langle #1,#2 \rangle }

\newcommand{\ext}{\wedge}
\newcommand{\campo}{\mathfrak{X}}
\newcommand{\zb}{\overline{z}}
\newcommand{\wb}{\overline{w}}
\newcommand{\bi}{\partial _{\infty}}

\begin{document}

\newtheorem{theorem}{Theorem}[section]
\newtheorem{proposition}[theorem]{Proposition}
\newtheorem{definition}[theorem]{Definition}
\newtheorem{lemma}[theorem]{Lemma}
\newtheorem{corollary}[theorem]{Corollary}
\newtheorem{remark}[theorem]{Remark}
\newtheorem{claim}[theorem]{Claim}

\newenvironment{proof}{\smallskip\noindent{\it Proof.}\hskip \labelsep}
                        {\hfill\penalty10000\raisebox{-.09em}{$\Box$}\par\medskip}

\title{When strictly locally convex hypersurfaces are embedded}

\author{Jos\'{e} M. Espinar\thanks{The author is partially
supported by Spanish MEC-FEDER Grant MTM2007-65249, and Regional J. Andaluc\'{i}a
Grants P06-FQM-01642 and FQM325}, Harold Rosenberg$\mbox{}^{\dag}$}
\date{}

\maketitle

\vspace{.4cm}

\noindent $\mbox{}^{\ast}$ D\'{e}partement de Math\'{e}matiques, Universit\'{e}
París-Est Marne-la-Vall\'{e}e, Cit\'{e} Descartes, 5 boulevard Descartes,
Champs-sur-Marne 77454, Marne-la-Vall\'{e}e, Cedex 2, France; e-mail:
jespinar@ugr.es\vspace{0.2cm}

\noindent $\mbox{}^{\dag}$ Instituto de Matematica Pura y Aplicada, 110 Estrada Dona
Castorina, Rio de Janeiro 22460-320, Brazil; e-mail: rosen@impa.br

\vspace{.3cm}

\begin{abstract}
In this paper we will prove Hadamard-Stoker type theorems in the following ambient
spaces: $\man ^n \times \r$, where $\man ^n $ is a $1/4-$pinched manifold, and
certain Killing submersions, e.g., Berger spheres and Heisenberg spaces. That is,
under the condition that the principal curvatures of an immersed hypersurfaces are
greater than some non-negative constant (depending on the ambient space), we prove
that such a hypersurface is embedded and we also study its topology.
\end{abstract}

\section{Introduction}

Hadamard proved a strictly compact locally convex hypersurface immersed in $\r ^n$
is an embedded sphere \cite{H}. Stoker then generalized this to complete immersed
strictly convex hypersurfaces in $\r ^n$: they are embedded spheres or $\r ^{n-1}$
\cite{S} Do Carmo and Warner \cite{CW} extended Hadamard's Theorem to $\s ^n$ and
$\h ^n$: such a compact hypersurface of $\s ^n$ is an embedded $\s ^{n-1}$ contained
in a hemisphere of $\s ^n$, and an embedded sphere in $\h ^n$. Currier extended
Stoker's Theorem to $\h ^n$, assuming all the principal curvatures are at least one
\cite{C}. S. Alexander \cite{A} proved Hadamard's Theorem in strict $-\kappa$
Hadamard manifolds assuming the principal curvatures are at least $-\kappa$.

We consider convexity in other ambient spaces; distinct from the space forms. The
Hadamard-Stoker Theorem was proved in $\hr$: a complete immersed surface in $\hr$ of
positive extrinsic curvature, is an embedded sphere or plane \cite{EGR}. Also, in
\cite{ES}, the authors generalized the above result to Killing submersions over a
strict Hadamard surface. For related results about locally convex hypersurfaces in
non-negatively curved manifolds see \cite{Esc}, and \cite{A} for locally convex
hypersurfaces in non-positively curved manifolds.

In this paper we will prove Hadamard-Stoker type theorems in the following ambient
spaces:

\begin{itemize}
\item $\man ^n \times \r$, where $\man ^n$ is a $1/4-$pinched manifold. We will see
that the $1/4-$pinched assumption is necessary (see Remark \ref{Rem:immersed}).
\item Certain Killing submersions, e.g., Berger spheres and Heisenberg spaces.
\end{itemize}

We begin Section \ref{Sect:14pinched} by studying the embeddedness of a family of
strictly convex hypersurfaces in $1/4-$pinched manifolds $\man ^n$, i.e. $\man ^n $
is a compact $n-$manifold whose sectional curvatures, $K_s$, are strictly positive.
Also, if $\kappa ^+ $ and $\kappa ^-$ denote the maximum and minimum of the
sectional curvatures on $\man ^n$ respectively, then, they verify $\kappa ^- /
\kappa ^+> 1/4$. More precisely, we prove (cf. \cite{Sc} for a relate use of this
idea):

\begin{quote}
{\bf Lemma \ref{Lem:emb}:} {\it Let $D^n$ and $\man ^n$ be $n-$dimensional
manifolds, $D^n$ compact with non-empty boundary $\Sigma$. Assume $g(t)$ and $h(t)$,
$0 \leq t \leq 1$, are continuous families of metrics on $D^n$ and $\man ^n $
respectively, and each $h(t)$ is $1/4-$pinched.

Let $f_t : (D^n , g(t)) \to (\man ^n , h(t))$ be isometric immersions, $0 \leq t
\leq 1$, continuous in $t$. Suppose $f_t (\Sigma):=\Sigma (t)$ has positive
principal curvatures for all $t$ (w.r.t. the normal pointing into $D^n$).

If $f_0$ is an embedding, then so is $f_t$ for all $t$.}
\end{quote}

Lemma \ref{Lem:emb} allows us to prove the following results in product spaces:

\begin{quote}
{\bf Theorem \ref{Theo:HScompact}:} {\it Let $\Sigma \subset \man^n  \times \r$ be a
locally strictly convex properly immersed connected hypersurface, where $\man ^n$ is
a $1/4-$pinched manifold. Then $\Sigma $ is properly embedded and homeomorphic to
the $n-$sphere or to the Euclidean $n-$space. In the later case, $\Sigma$ has either
a top end or a bottom end.}
\end{quote}

Also,

\begin{quote}
{\bf Theorem \ref{Theo:HSTorus}:} {\it Let $\Sigma \subset \man ^n \times \s ^1$ be
a complete immersed hypersurface whose principal curvatures are greater than $c$ at
any point of $\Sigma$. Assume also that $\man ^n$ is a $1/4-$pinched sphere, where
$\kappa ^- $ and $\kappa ^+$ denote the minimum and maximum of the sectional
curvatures of $\man ^n$ respectively. We normalize so that $\kappa ^+ =1$. If $c > 2
$ , then $\Sigma$ is an embedded sphere.}
\end{quote}

And for surfaces, we obtain:

\begin{quote}
{\bf Theorem \ref{Theo:SxR}:} {\it Let $\Sigma \subset \s ^2  \times \r$ be a
complete connected surface with constant positive extrinsic curvature. Then $\Sigma
$ is a rotational sphere in $\s^2 \times \r$.}
\end{quote}

We continue Section \ref{sect:HKSubmersion} considering strictly convex surfaces
immersed in a Hadamard-Killing submersion. We first establish the necessary tools we
will use in the proof of

\begin{quote}
{\bf Theorem \ref{Theo:HSzero}:} {\it Let $ \Sigma \subset \hm$ be a complete
connected immersed surface so that $k_{i}(p)
> |\tau (p)|$ for all $p\in \Sigma$, where $\hm $ is a Hadamard-Killing submersion.
Then $\Sigma$ is properly embedded. Moreover, $\Sigma$ is homeomorphic to
$\mathbb{S}^{2}$ or to $\mathbb{R}^{2}$. In the later case, when $\Sigma$ has no
point $p$ at which $N(p)$ is horizontal, $\Sigma$ is a Killing graph over a convex
domain of $\m$.}
\end{quote}

We should remark the the above Theorem \ref{Theo:HSzero} gives a Hadamard-Stoker
type Theorem in Heisenberg space.

Section \ref{Sect:Berger} is devoted to convex surfaces immersed in a Berger sphere.
Here, we prove

\begin{quote}
{\bf Theorem \ref{Theo:Berger}:} {\it Let $\Sigma \subset \s ^3 _B (\kappa , \tau)$
be a complete connected immersed surface so that $|k_i (p)|\geq \abs{\frac{\kappa -
4\tau ^2}{4\tau}}$ for all $p \in \Sigma$, here $k_i$, $i=1,2$, denotes the
principal curvatures of the immersion. Then, $\Sigma$ is embedded and homeomorphic
to a sphere.}
\end{quote}

Moreover, we will see how to prove Theorem \ref{Theo:HSzero} in the particular case
of Heisenberg space, using the techniques developed in Section \ref{Sect:Berger}

\section{$1/4-$pinched manifolds}\label{Sect:14pinched}

In this Section, we focus our attention on $1/4-$pinched manifolds.

\begin{definition}\label{Def:14pinched}
Let $\man ^n $ be a compact $n-$manifold whose sectional curvatures, $K_s$, are
strictly positive. Let $\kappa ^+ $ and $\kappa ^-$ denote the maximum and minimum
of the sectional curvatures on $\man ^n$ respectively. Then, we say that $\man $ is
$1/4-$pinched if $\kappa ^- / \kappa ^+ > 1/4$.
\end{definition}

First, we establish a Lemma about embeddedness of a family of closed strictly convex
submanifolds in a $1/4-$pinched manifolds, which will be the key result for
applications in what follows, and it is, in fact, of independent interest.

\begin{lemma}\label{Lem:emb}
Let $D^n$ and $\man ^n$ be $n-$dimensional manifolds, $D^n$ compact with non-empty
boundary $\Sigma$. Assume $g(t)$ and $h(t)$, $0 \leq t \leq 1$, are continuous
families of metrics on $D^n$ and $\man ^n $ respectively, and each $h(t)$ is
$1/4-$pinched.

Let $f_t : (D^n , g(t)) \to (\man ^n , h(t))$ be isometric immersions, $0 \leq t
\leq 1$, continuous in $t$. Suppose $f_t (\Sigma):=\Sigma (t)$ has positive
principal curvatures for all $t$ (w.r.t. the normal pointing into $D^n$).

If $f_0$ is an embedding, then so is $f_t$ for all $t$.
\end{lemma}
\begin{proof}
Since $D^n$ is compact, there exists $\delta >0$ such that $f_t$ is an embedding for
$0 \leq t < \delta$. It suffices to show $f _\delta$ is an embedding as well.

Suppose not, let $x ,y \in D^n$ be distinct points such that $f_\delta (x) =
f_\delta (y)$. If one of the points $\set{x,y}$ is not on $\Sigma$, then one can
find open neighborhoods of $x$ and $y$, $U_x$ and $V_y$, such that $U_x \cap V_y =
\emptyset$ and $f_\delta (U_x) \cap f_\delta (V_y) $ contains an open set of $\man
^n$. But then $f_t$ would not be an embedding for $t<\delta$, $t$ close to $\delta$;
a contradiction. Thus, both $x$ and $y$ are on $\Sigma$ and $f_{\delta | {\rm int}
D^n}$ is an embedding.

Consider $D^n$ with the metric $h(\delta ) = f_\delta ^* (g(\delta))$. Let $\beta$
be a minimizing geodesic of $(D^n , h(\delta))$ joining $x$ to $y$; $\beta$ exists
because $\Sigma$ is strictly convex. Set $l = {\rm Length}(\beta)$.

%Assume $\Sigma (t)$ becomes non-embedded. Let $|t_0 |\geq \epsilon$ be the first
%moment at which $\Sigma (t_0)$ is not embedded. Assume, without loss of generality,
%that $t_0$ is positive. Let $x(t),y(t) \in \Sigma (t)$, $0< t < t_0$, be sequences
%of points so that they converge to the same point in $\Sigma (t_0)$, that is, both
%sequences converge to the first point $p \in \Sigma (t_0)$ at which $\Sigma (t_0)$
%becomes non-embedded. Let $\Omega(t)$ denote the interior convex domain bounded by
%$\Sigma (t)$ and $\beta (t)$ be geodesic (strictly) contained in $\Omega (t)$
%joining $x(t)$ and $y(t)$.
%
%Then, $\beta (t)$ converges to a closed geodesic $\beta$ contained in $\Omega(t_0)$.
%Set $l = {\rm Length}(\beta)$.
%
%$q(t)$ be a sequence of points in the interior of $\Omega (t)$ which converges to an
%interior point $q_0$ at $\Omega (t_0)$. Let $\beta _x (t)$ and $\beta _y (t)$ be the
%geodesic arcs contained in $\Omega (t)$ joining $q(t)$ to $x(t)$ and $q(t)$ to $y
%(t)$ respectively.
%
%Set $\beta _x $ and $\beta _y$ the limit curves of $\beta _x (t)$ and $\beta _y (t)$
%respectively. Set $\eta = \beta _x \cup \beta _y$ and $l = {\rm Length}(\eta)$.

On the one hand, the injectivity radius of $\man ^n (\delta)$, ${\rm inj}(\man
^n(\delta))$, bounds $l$ from below as
$$  l/2 \geq {\rm inj}(\man ^n(\delta)) \geq \frac{\pi}{\sqrt{\kappa ^+(\delta)}} .$$

On the other hand, the Bonnet Theorem bounds $l$ from above as
$$ l \leq \frac{\pi}{\sqrt{\kappa ^-(\delta)}}. $$

Thus, joining the above inequalities, we obtain
$$ 2 \leq \frac{\sqrt{\kappa ^+(\delta)}}{\sqrt{\kappa ^-(\delta)}} ,$$that is
$$ \frac{\kappa ^-(\delta)}{\kappa ^+(\delta)} \leq 1/4 ,$$which contradicts the $1/4-$pinched
assumption. This proves the Lemma.
\end{proof}

\begin{remark}\label{Rem:14pinched}
The above $1/4-$pinched assumption is necessary as the next example shows. Let
$C(l)$ be the right cylinder of height $l$ and radius $1$ endowed with a flat
metric. Close it up with two spherical caps $S_i$, $i=1,2$, (one on the top and
another on the bottom) of radius $1$ endowed with its standard metric. Now, smooth
the surface $\man ^2 = C \cup S_1 \cup S_2 $ so that it is almost flat on the
cylinder and almost close to $1$ on the spherical caps, and has positive curvature.

So, if $l$ is large enough, it is not hard to see that we can consider a one
parameter family of strictly convex compact curves $\alpha (t)$ that are embedded
for $0< t < t_0$ and they became immersed for $t>t_0$. One only has to consider how
a family of concentric circles in $\r ^2$ becomes immersed on a cylinder as the
radius increases.
\end{remark}

\subsection{Applications}

First we consider strictly convex hypersurfaces $\Sigma$, i.e. all the principal
curvatures of $\Sigma$ are positive for a choice of a unit normal to $\Sigma$,
properly immersed in a product space $\man ^n \times \r$, where $\man ^n$ is a
$1/4-$pinched manifold.

\begin{theorem}\label{Theo:HScompact}
Let $\Sigma \subset \man^n  \times \r$ be a locally strictly convex properly
immersed connected hypersurface, where $\man ^n$ is a $1/4-$pinched manifold. Then
$\Sigma $ is properly embedded and homeomorphic to the $n-$sphere or to the
Euclidean $n-$space. In the later case, $\Sigma$ has either a top end or a bottom
end.
\end{theorem}

First we define a top or bottom end. Let $\man ^n \times \r $ be a product space and
$\Sigma$ a hypersurface in $\man ^n \times \r$. Let $\pi _{\rt} : \man ^n \times \r
\to \r$ be the usual projection. We denote by $h: \Sigma \to \r $ the {\it height
function}, that is, $h:= \left(\pi _{\rt}\right)_{| \Sigma}$.

\begin{definition}\label{Def:verticalend}
Let $\Sigma \subset \man ^n \times \r$ be a complete hypersurface. We say that
$\Sigma $ has a top end $E$ (resp. bottom end) if for any divergent sequence
$\set{p_n} \subset E$ the height function goes to $+\infty$ (resp. $-\infty$).
\end{definition}

{\it Proof of Theorem \ref{Theo:HScompact}.} Since $\Sigma$ is locally strictly
convex, the Gauss equation says that all the sectional curvatures of $\Sigma$ at any
point are positive. Thus, from Perelman's Soul Theorem \cite{P}, $\Sigma$ is either
compact or homeomorphic to $\r ^n$. In the latter case, $\Sigma$ has one topological
end $E$. $\man ^n$ is compact and $\Sigma$ is properly immersed so $E$ must go up or
down, otherwise $\Sigma \cap (\man ^n \times \set{0})$ would not be compact; so $E$
is a top or bottom end.

In a product space the leaves $\man ^n \times \set{t}$ are totally geodesic, hence
each connected component of $\Sigma \cap (\man ^n \times \set{t})$ is compact and
strictly convex when the intersection is transverse. Now, consider the foliation by
horizontal {\it hyperplanes} given by $P(t) := \man ^n \times \set{t}$ for $t \in \r
$. Since $\Sigma$ is either compact or has a top or bottom end, up to an isometry,
we can assume that $\Sigma \subset \man ^n \times [0 , + \infty)$ and $P(0)$ is the
horizontal hyperplane with the first contact point with $\Sigma$. At this point,
since $\Sigma$ is strictly convex, $\Sigma$ lies on one side of $P(0)$ and it is
(locally) a graph over a domain of $\man ^n$. Thus, there is $\epsilon >0$ so that
the hypersurfaces $C(t):=P(t) \cap U$ are embedded strictly convex hypersurfaces in
$\man ^n$ for $0<t<\epsilon$, where $U$ is the neighborhood of $\Sigma$ containing
the first contact point that can be expressed as a graph. Perhaps, $P(t)\cap \Sigma$
has other components distinct from $C(t)$ for $0<t<\epsilon$, but we only care how
$C(t)$ varies as $t$ increases. We also denote by $C(t)$ the continuous variation of
the submanifolds $P(t)\cap \Sigma$ when $t > \epsilon$.

Thus, it is easy to see that $C(t) $ either remains compact (non-empty) and embedded
for all $t>0$, or there exists $\bar t$ such that $C(t)$ are compact for all $0<t <
\bar t$, the component $C(t)$ disappears for $t > \bar t$ and $C(\bar t)$ is a
point. $C(t)$ remains embedded by Lemma \ref{Lem:emb}.

Thus, $\Sigma$ is either a properly embedded Euclidean $n-$space with a top end or
$\Sigma$ is an embedded $n-$sphere. \vspace{-0.85cm}
\begin{flushright}
$\square$
\end{flushright}

\begin{remark}\label{Rem:immersed}
Actually, the $1/4-$pinched assumption is necessary. consider the surface $\man ^2 =
C(l) \cup S_1 \cup S_2$ given in Remark \ref{Rem:14pinched}, with $l$ large enough
so that a family of concentric geodesic circles $S(r)$, $0< r< r_0$, in $\r ^2$
become immersed when we put them on the cylinder. That is, $S(r) \to {\rm point}$ as
$r\to 0$ and $S(r_0)$ is immersed and strictly convex in $C(l)$. Consider the
product space $\man ^2 \times \r$ and let $\Sigma := \bigcup _{0 \leq t \leq r_0}
(S(t) ,t) \cup \bigcup _{r_0 \leq t \leq 2r_0 } (S(2r_0 - t),t)$. Then, $\Sigma$ is
a strictly convex immersed surface in $\man ^2 \times \r$.
\end{remark}

Moreover, for strictly convex surfaces, from Bonnet and Gauss-Bonnet theorems, we
get

\begin{corollary}\label{Cor:Kcteproduct}
Let $\Sigma $ be a complete connected surface immersed in $\man ^2  \times \r$ with
extrinsic curvature bounded below by a positive constant, where $\man ^2$ is a
$1/4-$pinched surface. Then $\Sigma $ is an embedded sphere.
\end{corollary}

Anther application of Theorem \ref{Theo:HScompact} is the following

\begin{theorem}\label{Theo:HSTorus}
Let $\Sigma \subset \man ^n \times \s ^1$ be a complete immersed hypersurface whose
principal curvatures are greater than $c$ at any point of $\Sigma$. Assume also that
$\man ^n$ is a $1/4-$pinched sphere, where $\kappa ^- $ and $\kappa ^+$ denote the
minimum and maximum of the sectional curvatures of $\man ^n$ respectively. We
normalize so that $\kappa ^+ =1$. If $c > 2 $ , then $\Sigma$ is an embedded sphere.
\end{theorem}
\begin{proof}
First, since $\Sigma$ is complete and its principal curvatures are greater than a
positive constant, note that $\Sigma$ is compact by Bonnet's Theorem.

Now, lift $\Sigma$ to a compact hypersurface $\tilde \Sigma$ in the universal
covering space of $\man ^n \times \s ^1$, i.e. $\tilde \Sigma \subset \man ^n \times
\r$ is a compact hypersurface whose principal curvatures are greater than a positive
constant. Thus, from Theorem \ref{Theo:HScompact}, $\tilde \Sigma $ is an embedded
sphere in $\man ^n \times \r $.

Therefore, we can assume, up to an isometry, that $\tilde \Sigma \subset \man ^n
\times [0,+\infty)$ and $\man ^n \times \set{0}$ has a first contact point $p \in
\tilde \Sigma \cap \man ^n \times \set{0}$. Actually, $p$ is a global minimum.

Let $D$ be the geodesic disk in $\man ^n$ centered at $p$ of radius $ r:=\frac{\pi
}{ 2\sqrt{\kappa ^+}}-\epsilon$, $\epsilon >0$ small enough to be chosen. Note that
$D$ is (topologically) a $n-$ball and $S:=\partial D$ is strictly convex in $\man
^n$ with respect to the inward orientation. We claim that $\tilde \Sigma \subset D
\times [0,+\infty) $.

Set $C(t):= \tilde \Sigma \cap \left(\man ^n \times \set{t}\right)$, $t>0$. Then,
$C(t)$ is an embedded strictly convex $n-$sphere for $0< t < t_0$. For $t$ close to
$0$, $C(t)$ is contained in $D$. Assume there exists $\bar t \in (0 , t_0)$ so that
$C(\bar t) \cap S \not = \emptyset$. Set $q \in C(\bar t) \cap S$, then $d(p,q)\geq
r $, where $d(p,q)$ denotes the distance in $\tilde \Sigma$.

Now, from the Gauss equation, the sectional curvatures $\tilde K_s$ of $\tilde
\Sigma$ are bounded below by $\tilde K_s > c^2 $. So, the Bonnet Theorem bounds the
diameter of $\tilde \Sigma $ from above as
$$ {\rm diam }(\tilde \Sigma) < \pi / c.$$

Thus,
$$ \frac{\pi }{ 2\sqrt{\kappa ^+}} - \epsilon = r \leq d(p,q) \leq {\rm diam }(\tilde \Sigma) <
\pi / c,$$but, since $\kappa ^+ =1 $ and $c >2$, we can choose $\epsilon $ small
enough so that it contradicts the above inequality. Thus, $\tilde \Sigma \subset D
\times [0 ,+\infty)$.

Since $\tilde \Sigma \subset D \times \r $, we claim:

\begin{quote}
{\bf Claim 1:} {\it For any geodesic $\gamma \subset D$ joining two points in the
boundary $q_0,q_1 \in S$, if the {\it geodesic plane} (note that is not complete)
$P:= \gamma \times \r $ and $\tilde \Sigma $ intersect transversally,  then $\alpha
:= \tilde \Sigma \cap P$ is a strictly convex embedded Jordan curve in $P$.
Moreover, $\alpha $ has geodesic curvature greater than $c$.}
\end{quote}
{\it Proof of Claim 1:} Assume $P \cap \tilde \Sigma$ has two components (or more).
Let $C_1$ and $C_2$ denote such components. Since $\tilde \Sigma$ is an embedded
sphere, $C_i$, $i=1,2$, is a strictly convex embedded Jordan curve in $P$. Let $p_1
\in \Omega _1$ and $p_2 \in \Omega _2$ be points in the convex domains determined by
$C_1$ and $C_2$ in $P$ respectively. Let $\beta \subset \man ^n \times \r$ be the
geodesic joining $p_1$ and $p_2$, that is, $\beta$ is nothing but the {\it straight
line} in $P$ joining $p_1$ and $p_2$ (recall they are in the same vertical). Thus,
$\beta$ intersects $C_1$ and $C_2$ (note that $P$ is totally geodesic and flat in
$\man ^n \times \r$), which is a contradiction since $\tilde \Sigma$ is a strictly
convex embedded $n-$sphere.

Now, $\alpha $ has geodesic curvature greater than $c$, since the principal
curvatures of $\tilde \Sigma$ are greater than $c$ and $P$ is totally geodesic. This
proves Claim 1.
\begin{flushright}
$\square$
\end{flushright}

Now, we claim that $\Pi (\tilde \Sigma ) = \Sigma$ is embedded, here $\Pi : \man ^n
\times \r \to \man ^n \times \s^1 $ is the covering map. Assume $\Sigma$ is not
embedded, then, there exist two distinct points $p, q \in \tilde \Sigma$ that
project to the same point downstairs. Also, $p$ and $q$ are contained in the same
fiber in $\man ^n \times \r $ and their distance (along the fiber) has to be an
integer multiple of $1$. Now, let $\gamma $ be a geodesic in $D$ passing through
$\tilde p = \tilde q$, where $\tilde p$ and $\tilde q$ are the projections of $p$
and $q$ into $\man ^n$ respectively, so that $P := \gamma \times \r$ meets
transversally to $\tilde \Sigma$. Such a geodesic clearly exists.

Let $\alpha := \tilde \Sigma \cap P$ be the intersection curve, which is a simple
Jordan curve in $P$ with geodesic curvature greater than $c>2$ from Claim 1. So,
since $P$ is isometrically $\r ^2$, $\alpha $ is contained in a circle of radius
strictly less than $1/2$ in $P$. But, note that $p , q \in \alpha$ and the distance
from $p$ to $q$ is (at least) one. This is a contradiction. Therefore, $\Sigma$ is
embedded. This proves the result.
\end{proof}

Also, by using Theorem \ref{Theo:HScompact}, one can give an alternative, and more
geometric, proof of \cite[Theorem 7.3]{EGR} when $\man ^2 = \s^2$.

\begin{theorem}\label{Theo:SxR}
Let $\Sigma \subset \s ^2  \times \r$ be a complete connected surface with constant
positive extrinsic curvature. Then $\Sigma $ is a rotational sphere in $\s^2 \times
\r$.
\end{theorem}
\begin{proof}
From Theorem \ref{Theo:HScompact}, $\Sigma$ is an embedded sphere. So, we can assume
that $\Sigma \subset \s^2 \times (0 ,+ \infty)$. Do Alexandrov reflection w.r.t.
$P(t)= \s ^2 \times \set{t}$, $t>0$. Then, since $\Sigma$ is an embedded sphere,
there exists $t_0 >0$ so that $\Sigma$ is a bi-graph over $\s ^2 \times \set{t_0}$.
Up to an isometry we can assume $\Sigma$ is a bi-graph over $\s ^2 \times \set{0}$.

Set $\alpha = \Sigma \cap \s^2 \times \set{0}$, this curve is a strictly convex
simple Jordan curve, so, $\alpha $ is contained in some open hemisphere $\d $ of $\s
^2$ (see \cite{CW}). Let $\Omega$ be the compact domain bounded by $\alpha$. Since
$\Sigma$ is a bi-graph over $\overline \Omega$ and $\alpha$ is contained in an open
hemisphere $\d $, $\Sigma$ is contained in $\d \times \r $. Thus, \cite[Corollary
5.1]{CR} implies that $\Sigma$ is a rotational sphere.
\end{proof}

\section{Hadamard-Killing submersions}\label{sect:HKSubmersion}

In \cite{ES}, the authors studied locally strictly surfaces immersed in a strict
Hadamard-Killing submersion. We begin this Section reviewing the basis properties of
a Hadamard-Killing submersion (see \cite{ES} for details).

\subsection{On basic properties}

Most of this part in contained in \cite{ES}, but we need to introduce some concepts
and properties in order to make this paper self contained.

First, we start with Hadamard surfaces. For more details on Hadamard manifolds with
non positive sectional curvature see \cite{E}.

Let $\m $ be a Hadamard surface, that is, $\m$ is a complete, simply connected
surface with Gaussian curvature $\kappa \leq 0$.

It is well known that given two points $p,q \in \m$, there exists a unique geodesic
$\gamma_{pq}$ joining $p$ and $q$. We say that two geodesics $\gamma, \beta$ in $\m$
are asymptotic if there exists a constant $C>0$ such that $d(\gamma (t), \beta (t))$
$\leq$ $C$ for all $t$ $>$ 0. To be asymptotic is an equivalence relation on the
oriented unit speed geodesics or on the set of unit vectors of $\m$. We will denote
by $\gamma(+\infty)$ and $\gamma(-\infty)$ the equivalence classes of the geodesics
$t \rightarrow \gamma(t)$ and $t \rightarrow \gamma(-t)$ respectively. Moreover, an
equivalence class is called a point at infinity. $\m(\infty)$ denotes the set of all
points at infinity for $\m$ and $\m _{*} = \m\cup\m(\infty)$.

The set $\m _{*} = \m\cup \m(\infty)$ admits a natural topology, called the cone
topology, which makes $\m _{*}$ homeomorphic to the closed $2-$disk in $\r ^{2}$.

When $\m$ is a Hadamard surface with sectional curvature bounded above by a negative
constant then any two asymptotic geodesics $\gamma, \beta$ satisfy that the distance
between the two curves $\gamma_{|[t, +\infty)},\beta_{|[t, +\infty)} $ is zero for
any $t \in \r$. For each point $p \in \m$  and $x \in \m(\infty)$, there is a unique
geodesic $\gamma_{px}$ with initial condition $\gamma _{px}(0)=p$ and it is in the
equivalence class of $x$. For each point $p \in \m$ we may identify $\m(\infty)$
with the circle $\mathbb{S}^{1}$ of unit vectors in $T_{p}\m$ by means of the
bijection
\begin{equation*}
\begin{matrix}
G_p : & \s ^1 \subset T_p \m & \to & \m (\infty) \\
 & v & \longmapsto & \lim _{t\to +\infty} \gamma _{p,v}(t)
\end{matrix}
\end{equation*}where $\gamma _{p,v}$ is the geodesic with initial conditions $\gamma _{p,v}(0)=p$
and $\gamma _{p,v}'(0)=v$. In addition the hypothesis on the sectional curvature (it
is bounded above by a negative constant) yields there is an unique geodesic joining
two points of $\m(\infty)$.

Given a set $\Omega \subseteq \m$, we denote by $\partial_{\infty}\Omega$ the set
$\partial \Omega \cap \m(\infty)$,where $\partial \Omega$ is the boundary of
$\Omega$ for the cone topology. We orient $\m$ so that its boundary at infinity is
oriented counter-clockwise.

Let $\alpha$ be a complete oriented geodesic in $\m$, then
\begin{equation*}
\partial_{\infty}\alpha = \{\alpha^{-}, \alpha^{+}\}
\end{equation*}
where $\alpha^{-} = \lim _{t \to -\infty}\hspace{0,2cm}\alpha(t)$ and $\alpha^{+} =
\lim _{t \to + \infty}\hspace{0,2cm}\alpha(t)$. Here $t$ is arc length along
$\alpha$. We identify $\alpha$ with its boundary at infinity, writing $\alpha =
\{\alpha^{-}, \alpha^{+}\}$.

\begin{definition}\label{Def:orientedgeod}
Let $\theta_{1}$ and $\theta_{2}$ $\in \m(\infty)$, we define the oriented geodesic
joining $\theta_{1}$ and $\theta_{2}$, $\alpha (\theta_{1},\theta_{2})$, as the
oriented geodesic from $\theta_{1} \in \m(\infty)$ to  $\theta_{2} \in \m(\infty)$.
\end{definition}

\begin{definition}\label{Def:interiordomain}
Let $\alpha$ a oriented complete geodesic in $\m$. Let $J$ be the standard
counter-clockwise rotation operator. We call exterior set of $\alpha$ in $\m$,
$ext_{\mathbb{M}^{2}}(\alpha)$, the connected component of $\m \setminus \alpha$
towards which $J\alpha'$ points. The other connected component of $\m \setminus
\alpha$ is called the interior set of $\alpha$ in $\m$ and denoted by
$int_{\mathbb{M}^{2}}(\alpha)$.
\end{definition}

We continue with Riemannian submersions. Let $\man$ be a $3-$dimensional Riemannian
manifold so that it is a Riemannian submersion $\pi : \man \to \m$ over a surface
$(\m , g)$ with Gauss curvature $\kappa$, and the {\it fibers}, i.e. the inverse
image of a point at $\m $ by $\pi$, are the trajectories of a unit Killing vector
field $\xi $, and hence geodesics. Denote by $\meta{}{}$, $\camb$, $\ext $, $\bar R$
and $[,]$ the metric, Levi-Civita connection, exterior product, Riemann curvature
tensor and Lie bracket in $\man$, respectively. Moreover, associated to $\xi$, we
consider the operator $J: \campo (\man) \to \campo (\man)$ given by
\begin{equation*}
J X : = X \ext \xi , \, \, \, X \in \campo (\man).
\end{equation*}

Given $X \in \campo (\man)$, $X$ is {\it vertical} if it is always tangent to
fibers, and {\it horizontal} if always orthogonal to fibers. Moreover, if $X \in
\campo (\man) $, we denote by $X^v$ and $X^h$ the projections onto the subspaces of
vertical and horizontal vectors respectively.

One can see that, under these conditions, (see \cite[Proposition 2.6]{ES}) there
exists a function $\tau : \man \to \r $ so that
\begin{equation}\label{eq:tau}
\camb _X \xi = \tau \, X \ext \xi , \,
\end{equation}and then, it is natural to introduce the following definition:

\begin{definition}
A Riemannian submersion over a Hadamard surface $\m $, i.e., the Gaussian curvature
$\kappa$ of $\m$ is non-positive, whose fibers are the trajectories of a unit
Killing vector field $\xi$ will be called a {\it Hadamard-Killing submersion} and
denoted by $\hm $, where $\kappa $ is the Gauss curvature of $\m $ and $\tau $ is
given by \eqref{eq:tau}.
\end{definition}

Let $\Sigma \subset \hm$ be an oriented immersed connected surface. We endow
$\Sigma$ with the induced metric (\emph{First Fundamental Form}),
$\meta{}{}_{|\Sigma}$, in $\hm$, which we still denote by $\meta{}{}$. Denote by
$\nabla $ and $R$ the Levi-Civita connection and the Riemann curvature tensor of
$\Sigma$ respectively, and $S$ the shape operator, i.e., $S X = - \nabla _X N$ for
all $X \in \campo (\Sigma)$ where $N$ is the unit normal vector field along the
surface. Then $II(X,Y) = \meta{SX }{Y}$ is the \emph{Second Fundamental Form} of
$\Sigma$. Moreover, we denote by $J$ the (oriented) rotation of angle $\pi /2$ on
$T\Sigma$.

Set $\nu = \meta{N}{\xi}$ and $T = \xi - \nu N$, i.e., $\nu $ is the normal
component of the vertical field $\xi$, called the \emph{angle function}, and $T$ is
the tangent component of the vertical field.

In order to establish our result, we shall introduce some definitions and properties
about some particular surfaces in $\man (\kappa , \tau)$.

\begin{definition}\label{Def:cylinder}
We say that $\Sigma \subset \hm$ is a vertical cylinder over $\alpha$ if \,$\Sigma
:= \pi ^{-1} (\alpha)$, where $\alpha $ is a curve on $(\m  ,g)$. If $\alpha $ is a
geodesic, $\Sigma := \pi ^{-1}(\alpha)$ is called a vertical plane.
\end{definition}

One can check that a vertical plane is minimal, isometric to $\r ^2$ and its
principal curvature are bounded, in absolute value, by $|\tau (p)|$ at any point $p
\in \Sigma$ (see \cite[Proposition 2.10]{ES}).

We introduce a definition analogous to that given for complete geodesics in a
Hadamard surface since the notions of interior and exterior domains of a horizontal
oriented geodesic extend naturally to vertical planes.

\begin{definition}\label{Def:Interiorplane}
Let $\hm$ be a Hadamard-Killing submersion. For a complete oriented geodesic
$\alpha$ in $\m$ we call, respectively, interior and exterior of the vertical plane
$P = \pi^{-1}(\alpha)$ the sets
\begin{equation}
int_{\mathcal{M}(\kappa,\tau)}(P)=\pi^{-1}(int_{\mathbb{M}^{2}}(\alpha)),
\hspace{1cm}
ext_{\mathcal{M}(\kappa,\tau)}(P)=\pi^{-1}(ext_{\mathbb{M}^{2}}(\alpha)) \nonumber
\end{equation}
\end{definition}

Moreover, we will often use foliations by vertical planes of $\hm$. We now make this
precise.

\begin{definition}\label{Def:folitaionplanes}
Let $\hm$ be a Hadamard-Killing submersion. Let $P$ be a vertical plane in $\hm$,
and let $\beta(t)$ be an oriented horizontal geodesic in $\m $, with $t$ arc length
along $\beta$, $\beta(0)=p_{0} \in P$, $\beta'(0)$ orthogonal to $P$ at $p_{0}$ and
$\beta(t) \in ext_{\mathcal{M}(\kappa,\tau)}(P)$ for $t> 0$. We define the oriented
foliation of vertical planes along $\beta$, denoted by $P_{\beta(t)}$, to be the
vertical planes orthogonal to $\beta(t)$ with $P=P_{\beta}(0)$.
\end{definition}

To finish, we will give the definition of a particular type of curve in a vertical
plane. To do so, we recall a few concepts about Killing graphs in a Killing
submersion (see \cite{RST}).

Under the assumption that the fibers are complete geodesics of infinite length, it
can be shown (see \cite{St}) that such a fibration is topologically trivial.
Moreover, there always exists a global section
$$ s : \m \to \hm ,$$so, considering the flow $\phi _t$ of $\xi$, a trivialization
of the fibration is given by the diffeomorphism
$$ \begin{matrix}
\mr & \to & \hm \\
(p,t)& \rightarrowtail & \phi _t (s(p))
\end{matrix} $$

\begin{definition}\label{Def:Killinggraph}
Let $\pi : \hm \to \m $ be a Killing submersion. Let $\Omega \subset \m $ be a
domain. A Killing graph over $\Omega$ is a surface $\Sigma \subset \hm $ which is
the image of a section $s : \overline \Omega \to \hm $, with $s \in C^2(\Omega) \cap
C^0 (\overline \Omega)$. We may also consider graphs, $\Sigma \subset \hm$, without
boundary.
\end{definition}

Finally, we define:

\begin{definition}\label{Def:veticalcurve}
Let $P$ be a vertical plane in $\hm$ and $\gamma$ a complete embedded convex curve
in $P$. We say that $\alpha$ is an untilted curve in $P$ if there exists a point
$p\in \alpha$ so that $\phi _t (p)$ is contained in the convex body bounded by
$\alpha $ in $P$ for all $t>0$ (or $t<0$). Otherwise, we say that $\alpha $ is
tilted.
\end{definition}

\subsection{The result}

First, note that if $\Sigma \subset \hm $ is an immersed surface with positive
extrinsic curvature, then we can choose a globally defined unit normal vector field
$N$ so that the principal curvatures, i.e., the eigenvalues of the shape operator,
are positive. We denote them by $k_i$ for $i=1,2$.

We start with the following elementary result (see \cite[Proposition 3.1]{ES}).

\begin{proposition}\label{Pro:convex}
Let $\Sigma \subset \hm $ be an immersed surface whose principal curvatures satisfy
$k_{i}(p)> |\tau (p)|$ for all $p\in \Sigma$. Let $P$ be a vertical plane. If
$\Sigma$ and $P$ intersect transversally then each connected component $C$ of
$\Sigma \cap P$ is a strictly convex curve in $P$.
\end{proposition}

Now, we have the necessary tools for establishing our Theorem.

\begin{theorem}\label{Theo:HSzero}
Let $ \Sigma \subset \hm$ be a complete connected immersed surface so that $k_{i}(p)
> |\tau (p)|$ for all $p\in \Sigma$, where $\hm $ is a Hadamard-Killing submersion.
Then $\Sigma$ is properly embedded. Moreover, $\Sigma$ is homeomorphic to
$\mathbb{S}^{2}$ or to $\mathbb{R}^{2}$. In the later case, when $\Sigma$ has no
point $p$ at which $N(p)$ is horizontal, $\Sigma$ is a Killing graph over a convex
domain of $\m$.
\end{theorem}
\begin{proof}
As in \cite[Theorem 3.3]{ES}, we distinguish two cases depending on the existence of
a point $p$ in $\Sigma$ where $N(p)$ is horizontal.

\begin{quote}
{\bf Case 1:} {\it Suppose there is no point $p \in \Sigma$ where $N(p)$ is
horizontal. Then, $\Sigma $ is embedded and homeomorphic to the plane. Moreover, it
is a Killing graph over a convex domain in $\m$.}
\end{quote}

{\it Proof of Case 1:} It is the same as Case 1 in \cite[Theorem 3.3]{ES}.

\begin{quote}
{\bf Case 2:} {\it Suppose there is a point $p \in \Sigma$ so that $N(p)$ is
horizontal. Then, $\Sigma $ is embedded and homeomorphic to the sphere or to the
plane.}
\end{quote}

{\it Proof of Case 2:} By assumption $N$ is horizontal at $p$ and so, the tangent
plane $T_p \Sigma$ is spanned by $\set{\xi (p) , X(p)}$, where $X(p)$ is horizontal.
Set $\bar p := \pi (p)$ and $v := d\pi _p (X(p))$. Let $\alpha $ be the complete
geodesic in $\m$ with initial conditions $\alpha (0) = \bar p$ and $\alpha ' (0) =
v$. Set $P:=\pi ^{-1}(\alpha)$. Note that $p \in P \cap \Sigma$ and the principal
curvatures of $\Sigma$ at $p$ are greater than the principal curvatures of $P$ at
$p$, thus $\Sigma$ lies (locally around $p$) on one side of $P$. Without loss of
generality we can assume that $N(p)$ points to $ext_{\mathcal{M}(\kappa,\tau)}(P)$
(see Definition \ref{Def:interiordomain}), therefore, $\Sigma$ lies (locally around
$p$) in $ext_{\mathcal{M}(\kappa,\tau)}(P)$. Moreover, we parametrize the boundary
at infinity by $B: [0,2\pi] \to \m (\infty)$ so that $B(0)= \alpha ^-$, $B(\pi)=
\alpha ^+$ and $\partial _{\infty} ext_{\mathcal{M}(\kappa,\tau)}(P) = B([0,\pi])$.
Also, from now on, we identify the points at infinity with the points of the
interval $[0,2\pi]$.

Let $N_P$ be the unit normal vector field along $P$ pointing into
$ext_{\mathcal{M}(\kappa,\tau)}(P)$. Then, there exists neighborhoods $V \subset P$
and $U \subset \Sigma$ so that
$$ U := \set{ {\rm exp}_q (f(q)N_P(q)) : \, q \in V } ,$$where
$f:V \to \r $ is a smooth function and ${\rm exp}$ is the exponential map in $\hm$.

Let $P_{\beta}(t)$ be the foliation of vertical planes along $\beta$ (see Definition
\ref{Def:folitaionplanes}). From Proposition \ref{Pro:convex} and the fact that
locally $\Sigma$ is (in exponential coordinates) a graph, there is  $\epsilon > 0$
such that the curves $P_{\beta}(t) \cap U$ are embedded strictly convex curves (in
$P_{\beta}(t)$) for $0<t<\epsilon$. Perhaps, $P_{\beta}(t) \cap \Sigma$ has other
components distinct from $C(t)$ for each $0<t < \epsilon$, but we only care how
$C(t)$ varies as $t$ increases. We also denote by $C(t)$ the continuous variation of
the curves $P_{\beta}(t)\cap \Sigma$ when $t< \epsilon$.

Here, we also distinguish two cases:

\begin{quote}
{\bf Case A:} {\it If $C(t)$ remains compact for all $t>0$, then $\Sigma$ is
properly embedded and homeomorphic to the sphere or to the plane.}
\end{quote}

{\it Proof of Case A:} The proof is as Case A in \cite[Theorem 3.3]{ES}.

\begin{quote}
{\bf Case B:} {\it If $C(t)$ becomes non-compact, then $\Sigma$ is a properly
embedded plane.}
\end{quote}

{\it Proof of Case B:} First, note that Claim 1 and 2 in \cite[Theorem 3.3]{ES}
remain valid in this context with the same proof, i.e.,

\begin{quote}
\textbf{Claim 1:} {\it $C(\bar{t})$ is tilted (see Definition
\ref{Def:veticalcurve})}.
\end{quote}

\begin{quote}
\textbf{Claim 2:} {\it $\bi \pi (C(\bar{t}))$ is one point.}
\end{quote}

Thus, at this point, and following the notation above, we have: Let $P_{\beta}(t)$
be the foliation of vertical planes along $\beta$, where $P(0)$ is the vertical
plane over which $\Sigma$ is locally a graph at $p\in \Sigma$. Moreover, such a
graphical part of $\Sigma$ is contained in $ext_{\mathcal{M}(\kappa,\tau)}(P)$. Note
that $\beta (0) = \pi (p)$ and $\beta '(0) = d\pi _p (N(p))$.

Let $\gamma _t $ be the complete geodesic in $\m $ passing through $\beta (t)$ and
orthogonal to $\beta $ at $\beta (t)$. Set $P_{\beta} (\bar t) = \pi
^{-1}(\gamma_{\bar t})$, $\gamma _{\bar t}= \set{\gamma_{\bar t} ^- , \gamma _{\bar
t}^+}$, we parametrize the boundary at infinity by $B:[0,2\pi]\to \m (\infty)$ so
that $B(0)=\gamma _{\bar t}^-$, $B(\pi) = \gamma _{\bar t}^+$ and $\partial
_{\infty} int_{\mathcal{M}(\kappa,\tau)}(P_{\beta}(\bar t)) =B([0,\pi])$. Also, we
already know that $\tilde \Sigma _1:= \bigcup _{0 \leq t \leq \bar t} C(t) \subset
\Sigma$ is connected and embedded. By Claim 1, we may assume $\partial _{\infty}
C(\bar t) = \set{\gamma _{\bar t}^-}$.

Set $\epsilon >0$. Fix $t_{\epsilon}<\bar t$, close enough to $\bar t$, so that $\pi
(C(t_{\epsilon})) = \gamma_{t_\epsilon}([a , b])$ for some $a , b \in \r$ (recall
that $C(t_\epsilon)$ is compact).

Denote by $\Gamma _{\epsilon}(\theta) $ the complete geodesic in $\m$ passing
through $\gamma _{t_\epsilon} (r_\epsilon)$ and making an angle $\theta$ with
$\gamma _{t_{\epsilon}}$ at $\gamma _{t_\epsilon} (r_\epsilon)$, $0 \leq \theta \leq
\pi$. Fix $r_\epsilon < a$ so that $\tilde \Sigma _1 \subset
int_{\mathcal{M}(\kappa,\tau)}\left(\pi^{-1}\left( \Gamma _\epsilon
(\theta)\right)\right)$ for all $0 < \theta \leq \pi /2$. We orient $\Gamma
_{\epsilon}(\theta )$ so that $ \Gamma _\epsilon (0) ^- = \gamma _{t_{\epsilon}}^-$,
i.e., so that $\Gamma _\epsilon (\theta ) ^-$ moves away from $\gamma
_{t_\epsilon}^-$ as $\theta $ increase from $0$. Also, set $Q(\theta , \epsilon ):=
\pi ^{-1}\left( \Gamma _\epsilon (\theta) \right)$.

Now, $C(t_\epsilon)$ is a connected component of $\Sigma \cap Q(\theta , \epsilon)$,
we denote by $C'(\theta ,\epsilon)$ the continuous variation of the curves $\Sigma
\cap Q(\theta , \epsilon)$ when $\theta $ increase, recall that $C(t_\epsilon) =
C'(0 ,\epsilon )$. Since $C(t_\epsilon )$ is a compact embedded curve in the
vertical plane $Q(0 , \epsilon )$, there exists $\theta _0 > 0$ so that $C'(\theta ,
\epsilon )$ remains compact and embedded in $\tilde \Sigma \cap Q(\theta ,\epsilon)$
for all $0 < \theta < \theta _0$.

Now, we have the following two possibilities:

\begin{quote}
\textbf{(a)} {\it There exists $\epsilon >0$ so that $ C'(\theta ,\epsilon)$ remains
compact for all $ \theta$ satisfying $ B^{-1}\left( \Gamma _\epsilon (0)^+\right) <
B^{-1}\left(\Gamma _\epsilon (\theta)^+ \right) < 2\pi  $. }
\end{quote}

If this were the case, arguing as in Case B.1 in \cite[Theorem 3.3]{ES}, $\Sigma$ is
properly embedded and homeomorphic to the plane.

\begin{quote}
\textbf{(b)} {\it For all $\epsilon >0$ there exists $\theta _\epsilon $ so that
$C'(\theta _{\epsilon} ,\epsilon)$ becomes non-compact. }
\end{quote}

We will show that ${\bf (b)}$ is not possible. Letting $\epsilon \to 0$, we get the
existence of two distinct points on the boundary at infinity $\eta ^- < \eta ^+$ so
that $\Gamma _\epsilon (\theta _\epsilon) ^- \to \eta ^-$ and $\Gamma _\epsilon
(\theta _\epsilon) ^+ \to \eta ^+$ as $\epsilon \to 0$. Note that $\eta ^- = \gamma
_{\bar t}^-$. Set $\eta =\set{\eta ^- , \eta ^+}$ (see Definition
\ref{Def:orientedgeod}).

Let $T(s)$ be the foliation by vertical planes along a geodesic orthogonal to $\eta$
so that $T(0):= \pi ^{-1}(\eta)$. Take the orientation so that
$int_{\mathcal{M}(\kappa,\tau)}(T(0))=int_{\mathcal{M}(\kappa,\tau)}(\pi^{-1}(\eta))$.

By construction, $\tilde \Sigma \subset int_{\mathcal{M}(\kappa,\tau)}(T(0))$ where
$\tilde \Sigma = \Sigma _1 \cup \tilde \Sigma _2$, here $\tilde \Sigma _2$ is the
union of all the compact (embedded) components of $C(\theta , \epsilon)$ associated
to the continuous variation of $C(t_\epsilon)$. Moreover, $T(s)\cap \tilde \Sigma $
is either a compact embedded strictly convex curve, or a point or empty, for all $s
<0$. Set $\tilde C (s)$ the continuous variation of $\tilde \Sigma \cap T(s)$. Thus,
$\tilde C(0) = \lim _{s\to 0}\tilde C (s)$ should be an open embedded strictly
convex curve in $T(0)$ so that $\bi \pi (C(0)) = \set{\eta ^- , \eta ^+}$. But this
is impossible by Claim 2. So, ${\bf (b)}$ is proved.

This completes the proof of Theorem \ref{Theo:HSzero}

\end{proof}

\section{Berger spheres}\label{Sect:Berger}

For an approach to Berger spheres, we refer the reader to \cite{T}. We will recall
here only the necessary tools we will need, and for that, we follow \cite{T}. A
Berger sphere, denoted by $\s ^3_B (\kappa , \tau)$, is the usual three dimensional
sphere
$$ \s ^3 := \set{(z,w)\in \c ^2 \, : \, \, |z|^2+|w|^2 =1} ,$$endowed with the
metric
$$ \meta{X}{Y}_{(\kappa, \tau)} := \frac{4}{\kappa}\left( \meta{X}{Y}+
\left(\frac{4\tau ^2}{\kappa}-1\right)\meta{X}{V}\meta{Y}{V}\right) ,$$here
$\meta{}{}$ denotes the standard round metric on $\s ^3$, $V : \s ^3 \to \s ^3$ is
given by $V(z,w):=(i z , i w)$, and $\kappa >0$ and $\tau \neq 0$ are constants.
Moreover, $\s ^3_B (\kappa , \tau)$ is a model for the homogeneous space $\hmf$
described above when $\kappa >0$.

The vertical Killing field is $\xi := \frac{\kappa}{4\tau ^2}V$. Now, set $ E_1 (z,
w) := (- \wb , \zb) $ and $E_2 (z,w):=(-i \wb, i \zb)$. Then, $\set{E_1,E_2,V}$ is
an orthonormal basis of $T\s^3_B (\kappa, \tau)$ which satisfies $|E_i|^2 =
4/\kappa$, $i=1,2$, and $|V|^2=16\tau ^2 /\kappa$. Moreover, the connection $\nabla$
associated to $\meta{}{}_{(\kappa, \tau)}$ is given by:
\begin{equation}\label{eq:connectionBerger}
\begin{array}{lll}
\nabla _{E_1} E_1 = 0, & \nabla _{E_1} E_2 = -V, & \nabla _{E_1} V = \frac{4\tau
^2}{\kappa}E_2\\[3mm]
\nabla _{E_2} E_1 = V, & \nabla _{E_2} E_2 = 0, & \nabla _{E_1} V = -\frac{4\tau
^2}{\kappa}E_1\\[3mm]
\nabla _{V} E_1 = \left(\frac{4\tau^2}{\kappa}-1\right)E_2, & \nabla _{V} E_2 =
-\left(\frac{4\tau^2}{\kappa}-1\right)E_1, & \nabla _{V} V = 0
\end{array}
\end{equation}

First, we need to compute the principal curvatures of any equator of $\s ^3$ as
submanifold of $\s ^3 _B (\kappa , \tau)$. To do so, we only need to compute the
principal curvatures of the one parameter family of equators given by

$$ \psi (x,y)= (\cos \, x\sin \, y, \cos \, x\cos \, y, \sin \, x \sin \,\theta, \sin \, x\cos \,\theta),
$$where $\theta \in [0,\pi/2]$ is a constant. Any other equator is a rotation and/or a translation (w.r.t. the Berger
metric) of one in this family.

\begin{proposition}\label{Pro:equator}
Let $\psi : [0,2\pi]\times [0,2\pi] \to \s ^3 _B (\kappa , \tau)$ be an equator
given, for $\theta \in [0,\pi]$, by

$$ \psi (x,y)= (\cos \, x\sin \, y, \cos \, x\cos \, y, \sin \, x \sin \,\theta, \sin \, x\cos
\,\theta).$$

Then, it is minimal, i.e., $H=0$, and its extrinsic $K_e$ curvature is
$$ K_e := -\frac{4\tau ^2(\kappa - 4\tau^2)^2 \cos ^4 x}{(\kappa + 4\tau ^2 - (\kappa - 4\tau ^2)\cos \, 2x)^2}. $$

In particular, its principal curvatures $k_i$ are bounded in absolute value by
$$ |k_i| \leq \abs{\left(\frac{k}{4\tau ^2}-1\right)\tau} .$$
\end{proposition}

The proof of the above Proposition \ref{Pro:equator} will be given in Section
\ref{Sect:Proof}. Now, we have:

\begin{theorem}\label{Theo:Berger}
Let $\Sigma \subset \s ^3 _B (\kappa , \tau)$ be a complete connected immersed
surface so that $|k_i (p)|\geq \abs{\frac{\kappa - 4\tau ^2}{4\tau}}$ for all $p \in
\Sigma$, here $k_i$, $i=1,2$, denotes the principal curvatures of the immersion.
Then, $\Sigma$ is embedded and homeomorphic to a sphere.
\end{theorem}
\begin{proof}
First, note that $\Sigma$ is orientable by the assumptions on the principal
curvatures. Since the principal curvatures of the immersion are greater or equals
than any equator (see Proposition \ref{Pro:equator}), $\Sigma$ is locally on one
side of its tangent equator at each point (note that the intersection can be more
than one point, but, in any case, locally $\Sigma$ is at one side). Thus, if we
endow $\s ^3$ with the usual round metric, this means that $\Sigma$ has principal
curvatures greater or equals than zero at any point.

\begin{quote}
{\bf Claim 1:}{\it If $\Sigma \subset \s ^3 _ B (\kappa , \tau)$ is complete, then
$\Sigma \subset (\s ^3 , \meta{}{})$ is complete.}
\end{quote}

{\it Proof of Claim 1:} To see this, we can easily check that, for $X \in \campo (\s
^3)$, we have

\begin{equation*}
\begin{split}
\meta{X}{X}_{(\kappa , \tau)} & \leq \frac{4}{\kappa}\left( \norm{X}^2 +
\abs{\frac{4\tau ^2}{\kappa } -1} \meta{X}{V}^2\right) \\
 & \leq a ^2 \norm{X}^2 ,
\end{split}
\end{equation*}where $\norm{ \cdot }$ denotes the norm w.r.t. $\meta{ \cdot }{ \cdot }$, and
$$a^2 := \frac{4}{\kappa}\left( 1 + \abs{\frac{4\tau ^2}{\kappa}-1}\right).$$

This proves Claim 1.
\begin{flushright}
$\square$
\end{flushright}

That is, $\Sigma \subset (\s ^3 , \meta{}{})$ is a complete oriented connected
immersed surface whose principal curvatures are non-negative at any point. Then,
from \cite[Theorem 1.1]{CW}, $\Sigma$ is embedded and homeomorphic to a sphere.
Moreover, $\Sigma$ has to be contained in an open hemisphere. Note that, from
\cite[Theorem 1.1]{CW}, $\Sigma \subset (\s ^3 , \meta{}{})$ could be an equator,
but our original surface immersed in $\s ^3 _B (\kappa, \tau)$ is not (since both of
its principal curvatures are non-negative).

This finishes the proof.
\end{proof}

\subsection{A note on the Heisenberg space}

One can prove Theorem \ref{Theo:HSzero} in the particular case of Heisenberg space,
by using the same methods as in Theorem \ref{Theo:Berger}. Heisenberg space (see
\cite{D} for details), denoted by ${\rm Nil}_3 (\tau)$, is the usual $3-$dimensional
Euclidean space $\r ^3$ endowed with the metric
$$g_N := dx^2 + dy^2 + (\tau(y \, dx - x \, dy) +dz )^2 , $$where $(x,y,z)$ are the
standard coordinates in $\r ^3$, and $\tau \neq 0$.

Then, it is not hard to see that the principal curvatures $k^P_i$, $i=1,2$, of any
affine plane $P$, as a submanifold of ${\rm Nil}_3 (\tau)$, verify
$$ |k^P_i| \leq \tau , \, \, i=1,2.$$

Thus, if $\Sigma$ is a complete immersed surface whose principal curvatures are
greater than $\tau$ at any point, this implies that $\Sigma$ is locally on one side
of its tangent affine plane at that point. And so, it implies that $\Sigma \subset
(\r ^3 , g_0)$, where $g_0$ is the standard metric in the Euclidean space, is
locally strictly convex. Moreover, one can also check that a complete surface in
${\rm Nil}_3 (\tau)$ is complete in $\r ^3$. Thus, Stoker's Theorem \cite{S} implies
that $\Sigma$ is properly embedded and homeomorphic to the plane or to the sphere.

\section{Proof of Proposition \ref{Pro:equator}}\label{Sect:Proof}

Here, we include the proof of Proposition \ref{Pro:equator} for completeness. The
proof is based on tedious and straightforward computations.

First, we compute the orthogonal basis $\set{E_1, E_2, V}$ along $\psi$. It is easy
to check that
\begin{eqnarray*}
E_1 &=& (-\sin\, x\sin\, \theta, \sin\, x\cos\, \theta, \cos\, x\sin\, y, -\cos\,
x\cos\, y),\\[3mm]
E_2 &=& (-\sin\, x\cos\, \theta, -\sin\, x\sin\, \theta, \cos\, x\cos\, y, \cos\,
x\sin\, y),\\[3mm]
V&=& (-\cos\, x \cos \, y, \cos\, x \sin \, y, -\sin\, x\cos\,\theta, \sin\, x\sin\,
\theta)
\end{eqnarray*}

Second, we compute the partial derivatives of the immersion, which are given by:
\begin{eqnarray*}
\psi _x &=& (-\sin\, x \sin\, y, -\cos\, y \sin\, x, \cos\, x \sin\, \theta, \cos\, \theta \cos\, x), \\[3mm]
\psi _y &=& (\cos\, x \cos\, y, -\cos\, x \sin\, y, 0, 0).
\end{eqnarray*}

Now, we relate $\set{\psi _x , \psi _ y}$ in terms of $\set{E_1, E_2, V}$, that is:
\begin{eqnarray*}
\psi _x &=& -\cos(y+\theta)E_1 + \sin (y +\theta) E_2 ,\\[3mm]
\psi _y &=& -\frac{1}{2}\sin(2x)\sin(y+\theta)E_1
-\frac{1}{2}\sin(2x)\cos(y+\theta)E_2 - \cos ^2 x \, V .
\end{eqnarray*}

From the above equations, it is easy to see that the unit normal vector field is
given by
$$ N = -\alpha \left( \cos \, x \sin (y +\theta) E_1 + \cos \, x \cos (y +\theta) E_2
-\frac{\kappa}{4\tau ^2}\sin \, x \, V\right),$$ where
$$\alpha = \sqrt{\frac{2\kappa \tau ^2}{\kappa + 4\tau ^2 -(\kappa - 4\tau^2)\cos (2x)}}$$

The next step is to compute the covariant derivatives $\nabla _{\psi _x}\psi _x$,
$\nabla _{\psi _x}\psi _y = \nabla _{\psi _y} \psi _x$ and $\nabla _{\psi _y}\psi
_y$. To do so, we use \eqref{eq:connectionBerger} and the expressions of $\psi _x$
and $\psi _y$ in terms of $\set{E_1,E_2,V}$. So, we get:
\begin{eqnarray*}
\nabla _{\psi_x}\psi_x &=&0 \\[3mm]
\nabla _{\psi_x}\psi_y &=&\frac{(2\tau ^2- (\kappa - 2\tau ^2)\cos(2x))\sin \,
\theta \sin (y +\theta)}{4\alpha}E_1 \\
 & &+ \frac{(2\tau ^2- (\kappa - 2\tau ^2)\cos(2x))\sin \, \theta \cos (y
    +\theta)}{4\alpha}E_2 \\
 & &+ \frac{\kappa}{8\alpha}\sin \,\theta \sin (2x)V\\[3mm]
\nabla _{\psi_y}\psi_y &=&-\frac{(4\tau ^2- (\kappa - 4\tau ^2)\cos(2x))\sin \,
\theta \sin (2x) \cos (y +\theta)}{8\alpha}E_1 \\
 & &+ \frac{(4\tau ^2- (\kappa - 4\tau ^2)\cos(2x))\sin \,
\theta \sin (2x) \sin (y +\theta)}{8\alpha}E_2 .
\end{eqnarray*}

Thus, the coefficients of the first, $I$, and second, $II$, fundamental forms are
given by:
\begin{eqnarray*}
I(\psi _x , \psi _x) &=& \frac{4}{\kappa} \\
I(\psi _x , \psi _y) &=& 0 \\
I(\psi _x , \psi _x) &=& \frac{4\tau ^2}{\kappa \alpha ^2} \cos ^2 x
\end{eqnarray*}

\begin{eqnarray*}
II(\psi _x , \psi _x) &=& 0 \\
II(\psi _x , \psi _y) &=& 4\alpha (\kappa -4\tau ^2)\cos ^3 x  \\
II(\psi _x , \psi _x) &=& 0
\end{eqnarray*}

From the above expressions, we obtain that $H=0$ and the extrinsic curvature $K_e$
is given by
$$ K_e = - \frac{\alpha ^4 (\kappa -4\tau ^2)^2 \cos ^4 x}{\tau ^2 \kappa ^2} .$$

Since $H=0$ and the expression of the extrinsic curvature given above, we have

$$ |k_i| \leq \abs{\left( \frac{\kappa }{4\tau ^2}-1\right)\tau} ,$$where $k_i$,
$i=1,2$, are the principal curvatures. This finishes the proof of Proposition
\ref{Pro:equator}.

\end{document}